\documentclass[12pt]{article}

\usepackage[margin=1in]{geometry}
\usepackage{amsmath,amsthm,amssymb,amsfonts}
\usepackage{mathtools}
\usepackage{hyperref}
\usepackage[numbers,sort&compress]{natbib}
\usepackage{enumitem}
\usepackage{microtype}
\usepackage{tikz}

\theoremstyle{plain}
\newtheorem{theorem}{Theorem}[section]
\newtheorem{proposition}[theorem]{Proposition}
\newtheorem{lemma}[theorem]{Lemma}
\newtheorem{corollary}[theorem]{Corollary}

\theoremstyle{definition}
\newtheorem{definition}[theorem]{Definition}

\theoremstyle{remark}
\newtheorem{remark}[theorem]{Remark}

\newcommand{\R}{\mathbb{R}}
\newcommand{\C}{\mathbb{C}}

\newcommand{\tr}{\operatorname{tr}}
\newcommand{\doi}[1]{\textsc{doi:}~\texttt{#1}}
\DeclareMathOperator{\dt}{det}
\DeclareMathOperator{\rank}{rank}
\DeclareMathOperator{\MV}{MV}
\DeclareMathOperator{\Newt}{Newt}
\DeclareMathOperator{\ordop}{ord}
\newcommand{\ip}[2]{\langle #1,\, #2\rangle}
\newcommand{\Hh}{\widehat{H}}

\title{An intrinsic characterization of the Bogdanov--Takens
	normal-form coefficients and a mixed-volume obstruction to
	non-isolated degeneracies}

\author{
	V. Castellanos
	\and
	E.~Chan-L\'opez\thanks{Corresponding author.}
}

\date{%
	Divisi\'on Acad\'emica de Ciencias B\'asicas,
	Universidad Ju\'arez Aut\'onoma de Tabasco, 86690 Cunduac\'an,
	Tabasco, Mexico}

\begin{document}
	\maketitle
	
	\begin{abstract}
		Let $X$ be a planar vector field with an equilibrium $p$ at which the
		Jacobian $J=DX(p)$ is nilpotent of rank one, and let $q_{0}$ span
		$\ker J$.  We prove that the two coefficients $a$ and $b$ of the
		Bogdanov--Takens (BT) normal form are the directional derivatives,
		along $q_{0}$, of the two invariants of the Jacobian:
		\[
		a=-\tfrac12\,\ip{\nabla\!\dt DX(p)}{q_{0}},
		\qquad
		b=\ip{\nabla\!\tr DX(p)}{q_{0}} .
		\]
		The identity is invariant under changes of phase-space coordinates and
		equivariant under the rescaling of $q_{0}$, and the resulting formula
		requires neither generalized eigenvectors nor the second-order
		multilinear form.  It
		yields a coordinate-free reading of the BT nondegeneracy conditions in
		terms of the kernel line field of the projection of the equilibrium
		manifold onto parameter space, the transformation rule
		$(a,b)\mapsto(h^{2}a,hb)$ under orbital equivalence, and the fact that
		$a=0$ whenever the vector field factors through a function vanishing at
		$p$.
		
		We then prove an obstruction of a combinatorial nature.  For a
		Kolmogorov system $\dot x=xA/g_{1}$, $\dot y=yB/g_{2}$ and an
		equilibrium $p$ in the torus $(\C^{*})^{2}$ at which the Jacobian is
		nilpotent and nonzero, the order $m=\ordop f$ in the Takens normal form
		is bounded by the mixed volume of the Newton polytopes of $A$ and $B$.  In
		particular, if $\MV(\Newt A,\Newt B)\le2$ then $a\neq0$ unless the
		equilibrium fails to be isolated, and the nilpotent singularities of
		saddle, focus and elliptic type are unreachable: only cusp-type
		singularities occur at isolated equilibria, while their exact
		codimension is not controlled by the mixed-volume bound.  The class of systems with cross-product cubic
		terms, for which the mixed volume equals $2$, is treated in detail, and
		two classical Bazykin models are shown to be instances.
	\end{abstract}
	
	

	\section{Introduction}\label{sec:intro}
	
	The Bogdanov--Takens bifurcation \cite{Bogdanov1975,Takens1974} is the
	generic codimension-two local bifurcation of a planar equilibrium with
	a nilpotent linear part.  Its normal form on the center manifold is
	\begin{equation}\label{eq:BT}
		\dot u = v,\qquad
		\dot v = a\,u^{2}+b\,u\,v+o(|(u,v)|^{2}),
	\end{equation}
	and the bifurcation is nondegenerate when $ab\neq0$
	\cite{Kuznetsov2004}.  The two ways of losing nondegeneracy are not
	equivalent.  In the classification of nilpotent singularities
	\cite{Dumortier1977,DLA2006}, writing the germ in Takens form
	$\dot u=v$, $\dot v=f(u)+v\,g(u)$ with $m=\operatorname{ord}f$ and
	$n=\operatorname{ord}g$, the BT point is $(m,n)=(2,1)$.  The degeneracy
	$b=0$ with $a\neq0$ gives $(m,n)=(2,2)$, a singularity that remains
	topologically a cusp and whose generic three-parameter unfolding is the
	object of \cite{DRS1987}.  The degeneracy $a=0$ with $b\neq0$ gives
	$m\geq3$, a singularity of saddle, focus or elliptic type, unfolded in
	\cite{DRSZ1991}.
	
	This paper has two parts.
	
	The first is an identity.  It seems not to have been recorded, and it
	is the source of everything else.
	
	\medskip\noindent
	\textbf{Theorem A (Section~\ref{sec:intrinsic}).}
	\emph{Let $X$ be a planar vector field of class $C^{2}$ with $X(p)=0$
		and $J=DX(p)$ nilpotent of rank one, let $q_{0}$ span $\ker J$, and let
		$a,b$ be the coefficients of \eqref{eq:BT} computed from a basis
		adapted to $q_{0}$.  Then}
	\begin{equation}\label{eq:ThmA}
		a=-\tfrac12\,\ip{\nabla\!\dt DX(p)}{q_{0}},
		\qquad
		b=\ip{\nabla\!\tr DX(p)}{q_{0}} .
	\end{equation}
	
	The two invariants of a $2\times2$ matrix are its trace and its
	determinant; the two BT coefficients are their derivatives along the
	kernel of the linearization.  Formula~\eqref{eq:ThmA} involves no
	generalized eigenvectors, no biorthonormalization and no second-order
	multilinear form; it is invariant under changes of phase-space
	coordinates (Lemma~\ref{lem:invRHS}) and equivariant under the
	rescaling of $q_{0}$, which is exactly the residual freedom left by the
	normal-form construction (Lemma~\ref{lem:scaling}).
	
	Theorem~A converts the analytic nondegeneracy conditions into
	statements about a line field.  Let $X_{\lambda}$ be a family,
	$E=\{(x,\lambda): X(x,\lambda)=0\}$ the equilibrium manifold,
	$\pi\colon E\to\R^{r}$ the projection onto parameter space,
	$\Sigma=\{\dt D_{x}X=0\}\cap E$ the fold set and
	$\Hh=\{\tr D_{x}X=0\}\cap E$ the neutrality set.  Along $\Sigma$ the
	kernel of $d\pi$ restricted to $TE$ is a line field $\ell$.
	
	\medskip\noindent
	\textbf{Theorem B (Section~\ref{sec:geometry}).}
	\emph{At a BT point, $a=0$ if and only if $\ell\subset T\Sigma$, and
		$b=0$ if and only if $\ell\subset T\Hh$; the latter holds if and only
		if $\pi|_{\Hh}$ fails to be an immersion.}
	
	We stress in Remark~\ref{rem:nottransversality} that $b=0$ is
	\emph{not} a failure of transversality between $\Sigma$ and $\Hh$
	inside $E$.
	
	The second part is an obstruction.  Theorem~A has an immediate
	consequence: if $X=h\,Z$ with $h(p)=0$, then $\dt DX$ vanishes
	identically on $\{h=0\}$, whose tangent at $p$ is $\ker J$, and
	therefore $a=0$ (Proposition~\ref{prop:product}).  Non-isolated
	equilibria always produce $a=0$.  The content of the second part is a
	converse for a wide class of systems, and the mechanism is
	combinatorial.
	
	\medskip\noindent
	\textbf{Theorem C (Section~\ref{sec:mult}).}
	\emph{Let $\dot x=x\,A/g_{1}$, $\dot y=y\,B/g_{2}$ with $A,B,g_{1},
		g_{2}$ Laurent polynomials, and let $p\in(\C^{*})^{2}$ be an
		equilibrium with $g_{1}(p)g_{2}(p)\neq0$ and $DX(p)$ nilpotent and
		nonzero.  If $A$ and $B$ have a common nonconstant factor vanishing at
		$p$, then $a=0$ and $p$ is not isolated.  Otherwise $p$ is isolated and}
	\[
	\operatorname{ord}f
	=\dim_{\C}\C[[x,y]]/(X_{1},X_{2})
	\le \MV(\Newt A,\Newt B).
	\]
	\emph{In particular, if $\MV(\Newt A,\Newt B)\le2$ then $a\neq0$ and the
		singularity is topologically a cusp, whatever its codimension; the
		nilpotent saddle, focus and elliptic strata of \cite{DRSZ1991} are
		unreachable.}
	
	The proof is short: the multiplicity of a nilpotent zero equals
	$\operatorname{ord}f$, which is $2$ exactly when $a\neq0$
	(Lemma~\ref{lem:mult}); and the multiplicity of an isolated zero in the
	torus is bounded by the mixed volume of the Newton polytopes
	(Lemma~\ref{lem:BKK}), a consequence of the theorem of Bernstein and
	Kushnirenko \cite{Bernstein1975,Kushnirenko1976} together with the
	conservation of multiplicity under small perturbations.  We emphasize that the bound counts only
	solutions in the torus, so that solutions on the axes and at infinity
	are irrelevant; this is what makes the argument short.
	
	Section~\ref{sec:family} applies Theorem~C to the class
	\[
	\dot x=x\,P(x,y),\qquad \dot y=y\,Q(x,y),
	\]
	with $P$ and $Q$ affine-bilinear, that is, with cubic terms restricted
	to the cross-products $x^{2}y$ and $xy^{2}$.  Here $\Newt P=\Newt Q$ is
	the unit square and $\MV=2$, so that $a=0$ occurs exactly at the
	parameter values for which the equilibrium ceases to be isolated; these
	are identified explicitly.  Section~\ref{sec:models} treats Bazykin's
	models with predator satiation and with mutual interference
	\cite{Bazykin1998}.  Both realize the degeneracy $b=0$, $a\neq0$, and
	only that one.

	\section{The BT coefficients as directional derivatives}
	\label{sec:intrinsic}
	
	\subsection{Setting}
	
	Let $X$ be a $C^{2}$ vector field on an open subset of $\R^{2}$ with
	$X(p)=0$ and $J:=DX(p)$ nilpotent with $\rank J=1$.  Write
	$B(u,v):=D^{2}X(p)[u,v]$.
	
	\begin{definition}\label{def:kbasis}
		A \emph{Kuznetsov basis adapted to $q_{0}$} is a quadruple
		$(q_{0},q_{1},p_{0},p_{1})$ of vectors of $\R^{2}$ with
		\begin{equation}\label{eq:kbasis}
			Jq_{0}=0,\quad Jq_{1}=q_{0},\quad
			J^{\top}p_{1}=0,\quad J^{\top}p_{0}=p_{1},
		\end{equation}
		\begin{equation}\label{eq:biorth}
			\ip{q_{0}}{p_{0}}=\ip{q_{1}}{p_{1}}=1,\qquad
			\ip{q_{0}}{p_{1}}=\ip{q_{1}}{p_{0}}=0 .
		\end{equation}
		The associated BT coefficients are
		\begin{equation}\label{eq:abdef}
			a=\tfrac12\ip{p_{1}}{B(q_{0},q_{0})},\qquad
			b=\ip{p_{0}}{B(q_{0},q_{0})}+\ip{p_{1}}{B(q_{0},q_{1})} .
		\end{equation}
	\end{definition}
	
	These are the quantities computed in \cite{Kuznetsov2004,Kuznetsov2005}.
	
	\begin{lemma}[Existence]\label{lem:exists}
		If $J$ is nilpotent with $\rank J=1$ and $q_{0}$ spans $\ker J$, a
		Kuznetsov basis adapted to $q_{0}$ exists, and $(p_{0},p_{1})$ is
		exactly the basis dual to $(q_{0},q_{1})$.
	\end{lemma}
	
	\begin{proof}
		From $J^{2}=0$ we get $\operatorname{Im}J\subset\ker J$, and both are
		one-dimensional, so $q_{0}\in\operatorname{Im}J$ and some $q_{1}$ with
		$Jq_{1}=q_{0}$ exists; $q_{1}\notin\ker J$, so $(q_{0},q_{1})$ is a
		basis of $\R^{2}$.  Let $(p_{0},p_{1})$ be the dual basis, which is
		\eqref{eq:biorth}.  Then
		$\ip{q_{0}}{J^{\top}p_{1}}=\ip{Jq_{0}}{p_{1}}=0$ and
		$\ip{q_{1}}{J^{\top}p_{1}}=\ip{q_{0}}{p_{1}}=0$ give $J^{\top}p_{1}=0$;
		and $\ip{q_{0}}{J^{\top}p_{0}}=0=\ip{q_{0}}{p_{1}}$ together with
		$\ip{q_{1}}{J^{\top}p_{0}}=\ip{q_{0}}{p_{0}}=1=\ip{q_{1}}{p_{1}}$ give
		$J^{\top}p_{0}=p_{1}$.  Conversely, \eqref{eq:biorth} says precisely
		that $(p_{0},p_{1})$ is dual to $(q_{0},q_{1})$.
	\end{proof}
	
	\begin{lemma}[Residual freedom]\label{lem:scaling}
		Given $q_{0}\neq0$ in $\ker J$, the quadruple satisfying
		\eqref{eq:kbasis}--\eqref{eq:biorth} is determined up to
		$(q_{0},q_{1},p_{0},p_{1})\mapsto(q_{0},q_{1}+dq_{0},p_{0}-dp_{1},p_{1})$,
		$d\in\R$, and the coefficients \eqref{eq:abdef} do not depend on $d$.
		Replacing $q_{0}$ by $cq_{0}$, $c\neq0$, forces
		$(q_{1},p_{0},p_{1})\mapsto(cq_{1},c^{-1}p_{0},c^{-1}p_{1})$ and
		produces $(a,b)\mapsto(ca,cb)$.
	\end{lemma}
	
	\begin{proof}
		Both statements are direct verifications from
		\eqref{eq:kbasis}--\eqref{eq:abdef}.  For the first, substituting
		$q_{1}+dq_{0}$ and $p_{0}-dp_{1}$ in \eqref{eq:abdef} adds
		$-d\ip{p_{1}}{B(q_{0},q_{0})}+d\ip{p_{1}}{B(q_{0},q_{0})}=0$ to $b$ and
		leaves $a$ unchanged.  For the second, $a$ and $b$ are homogeneous of
		degree $2$ in the $q$'s and $-1$ in the $p$'s.
	\end{proof}
	
	Thus $a$ and $b$ are attached to the pair $(X,q_{0})$, and only their
	vanishing, together with the sign of $ab$, is independent of the choice
	of $q_{0}$.  Formula~\eqref{eq:ThmA} is homogeneous of degree one in
	$q_{0}$, hence compatible with Lemma~\ref{lem:scaling}.
	
	\subsection{Invariance}
	
	Let $\varphi$ be a $C^{2}$ diffeomorphism with $\varphi(z_{0})=p$, and
	let $Y:=(D\varphi)^{-1}\,X\circ\varphi$.  Put $Q:=D\varphi(z_{0})$ and
	$C:=D^{2}\varphi(z_{0})$.  Vectors and covectors are transported by
	\begin{equation}\label{eq:transport}
		q\longmapsto q^{z}=Q^{-1}q,\qquad
		\rho\longmapsto \rho^{z}=Q^{\top}\rho ,
	\end{equation}
	which preserves the pairing.
	
	\begin{lemma}[Covariance of the normal-form data]\label{lem:invLHS}
		With the notation above, $J_{Y}=Q^{-1}JQ$,
		\begin{equation}\label{eq:Btransf}
			B_{Y}(u,v)=Q^{-1}\bigl[\,J\,C(u,v)+B(Qu,Qv)
			-C(u,J_{Y}v)-C(v,J_{Y}u)\,\bigr],
		\end{equation}
		the transported quadruple \eqref{eq:transport} is a Kuznetsov basis for
		$Y$ adapted to $q_{0}^{z}$, and $a_{Y}=a$, $b_{Y}=b$.
	\end{lemma}
	
	\begin{proof}
		Expanding $\varphi(z_{0}+\zeta)=p+Q\zeta+\tfrac12C(\zeta,\zeta)+o(|\zeta|^{2})$
		and
		$(D\varphi)^{-1}=Q^{-1}-Q^{-1}C(\zeta,\cdot)Q^{-1}+o(|\zeta|)$ in
		$Y(z_{0}+\zeta)=(D\varphi)^{-1}X(\varphi(z_{0}+\zeta))$ and collecting
		terms of order one and two gives $J_{Y}=Q^{-1}JQ$ and
		\eqref{eq:Btransf} after polarization.  The relations
		\eqref{eq:kbasis}--\eqref{eq:biorth} for the transported quadruple are
		immediate from $J_{Y}=Q^{-1}JQ$ and the invariance of the pairing.
		
		For the coefficients, use $J_{Y}q_{0}^{z}=0$, $J_{Y}q_{1}^{z}=q_{0}^{z}$,
		$\ip{Q^{\top}\rho}{Q^{-1}w}=\ip{\rho}{w}$,
		$\ip{p_{1}}{Jw}=\ip{J^{\top}p_{1}}{w}=0$ and
		$\ip{p_{0}}{Jw}=\ip{p_{1}}{w}$.  From \eqref{eq:Btransf},
		\[
		a_{Y}=\tfrac12\ip{p_{1}}{J\,C(q_{0}^{z},q_{0}^{z})}
		+\tfrac12\ip{p_{1}}{B(q_{0},q_{0})}=a ,
		\]
		and, writing $\kappa:=C(q_{0}^{z},q_{0}^{z})$,
		\[
		b_{Y}=\bigl[\ip{p_{1}}{\kappa}+\ip{p_{0}}{B(q_{0},q_{0})}\bigr]
		+\bigl[\ip{p_{1}}{B(q_{0},q_{1})}-\ip{p_{1}}{\kappa}\bigr]=b . \qedhere
		\]
	\end{proof}
	
	\begin{lemma}[Invariance of the right-hand side]\label{lem:invRHS}
		With the notation above,
		\[
		\ip{\nabla\!\tr DY(z_{0})}{q_{0}^{z}}=\ip{\nabla\!\tr DX(p)}{q_{0}},
		\qquad
		\ip{\nabla\!\dt DY(z_{0})}{q_{0}^{z}}=\ip{\nabla\!\dt DX(p)}{q_{0}} .
		\]
	\end{lemma}
	
	\begin{proof}
		Write $DY(z)=M(z)+R(z)$ with
		$M(z)=D\varphi(z)^{-1}DX(\varphi(z))D\varphi(z)$ and
		$R(z)\,\zeta=\bigl[D_{\zeta}(D\varphi^{-1})\bigr]X(\varphi(z))$.  Then
		$\tr M(z)=\tr DX(\varphi(z))$ and $\dt M(z)=\dt DX(\varphi(z))$, and
		$R(z_{0})=0$ because $X(p)=0$.  Differentiating $R$ at $z_{0}$ along
		$q_{0}^{z}$,
		\[
		\bigl[dR(q_{0}^{z})\bigr]\zeta
		=\bigl[D^{2}(D\varphi^{-1})(q_{0}^{z},\zeta)\bigr]X(p)
		+\bigl[D_{\zeta}(D\varphi^{-1})\bigr]\bigl(DX(p)\,Q\,q_{0}^{z}\bigr)=0 ,
		\]
		since $X(p)=0$ and $Qq_{0}^{z}=q_{0}\in\ker DX(p)$.  Hence
		$d(\tr R)(q_{0}^{z})=0$, and, using
		$\dt(M+R)=\dt M+\tr(\operatorname{adj}(M)R)+\dt R$ together with
		$R(z_{0})=0$ and $dR(q_{0}^{z})=0$,
		$d(\dt DY)(q_{0}^{z})=d(\dt M)(q_{0}^{z})$.  The chain rule finishes
		the proof.
	\end{proof}
	
	The hypothesis $q_{0}\in\ker DX(p)$ is essential: neither $\tr DX$ nor
	$\dt DX$ is invariant at points that are not equilibria, and it is
	precisely the vanishing of $DX(p)q_{0}$ that kills the corrections.
	
	\subsection{Proof of Theorem A}
	
	\begin{theorem}[Theorem A]\label{thm:A}
		Let $X$ be of class $C^{2}$ with $X(p)=0$ and $J=DX(p)$ nilpotent of
		rank one, let $q_{0}$ span $\ker J$ and let $a,b$ be as in
		Definition~\ref{def:kbasis}.  Then \eqref{eq:ThmA} holds.
	\end{theorem}
	
	\begin{proof}
		Work in the adapted basis $(q_{0},q_{1})$ of
		Definition~\ref{def:kbasis}, with dual basis $(p_{0},p_{1})$, so that
		\[
		[J]_{(q_{0},q_{1})}=\begin{pmatrix}0&1\\ 0&0\end{pmatrix}.
		\]
		Let $L:=D(DX)_{p}[q_{0}]$ be the derivative of the matrix-valued map
		$z\mapsto DX(z)$ at $p$ in the direction $q_{0}$; by definition of the
		second-order form, $L(w)=B(q_{0},w)$.  In the same basis,
		\[
		[L]=\begin{pmatrix}
			\ip{p_{0}}{B(q_{0},q_{0})} & \ip{p_{0}}{B(q_{0},q_{1})}\\[2pt]
			\ip{p_{1}}{B(q_{0},q_{0})} & \ip{p_{1}}{B(q_{0},q_{1})}
		\end{pmatrix}.
		\]
		
		Since the trace is linear,
		\[
		d(\tr DX)_{p}(q_{0})=\tr L
		=\ip{p_{0}}{B(q_{0},q_{0})}+\ip{p_{1}}{B(q_{0},q_{1})}=b,
		\]
		which is \eqref{eq:abdef}.  For the determinant, differentiating at the
		matrix $J$ gives $d(\dt)_{J}(L)=\tr\bigl(\operatorname{adj}(J)\,L\bigr)$,
		and
		$\operatorname{adj}\left(\begin{smallmatrix}0&1\\0&0\end{smallmatrix}\right)
		=\left(\begin{smallmatrix}0&-1\\0&0\end{smallmatrix}\right)$, so only
		the $(2,1)$ entry of $[L]$ survives:
		\[
		d(\dt DX)_{p}(q_{0})=-\ip{p_{1}}{B(q_{0},q_{0})}=-2a .
		\]
		Both identities are \eqref{eq:ThmA}.
	\end{proof}
	
	\begin{remark}
		The computation uses only the $2$-jet of $X$ at $p$ and never invokes a
		normal form, which is why $C^{2}$ regularity suffices and why no
		remainder term appears.  Lemmas~\ref{lem:scaling},
		\ref{lem:invLHS} and~\ref{lem:invRHS} are not needed for the proof;
		they are retained because they give the separate covariance
		interpretation of the two sides, which is what
		Section~\ref{sec:geometry} uses.
	\end{remark}
	
	\begin{remark}\label{rem:regularity}
		Both sides of \eqref{eq:ThmA} depend only on the $2$-jet of $X$ at $p$,
		so $C^{2}$ regularity suffices.  For a merely $C^{2}$ field Taylor's
		theorem gives a remainder $o(|(u,v)|^{2})$ and not
		$\mathcal O(|(u,v)|^{3})$, which is why \eqref{eq:BT} is written with
		the former; class $C^{3}$ is needed only to interpret $a$ and $b$ as
		the coefficients of a germ smoothly equivalent to \eqref{eq:BT} up to
		order three, and it is then legitimate to write
		$\mathcal O(|(u,v)|^{3})$.  If one prefers a normalizing map, it may be
		taken to be a polynomial of degree two: a linear map sending $q_{0}$ to
		$(1,0)^{\top}$ and putting $J$ in Jordan form, followed by a
		near-identity quadratic change removing all quadratic terms except
		$u^{2}$ and $uv$ in the second component.
	\end{remark}
	
	\subsection{Two consequences}
	
	\begin{corollary}[Orbital equivalence]\label{cor:orbital}
		Let $h$ be a positive $C^{2}$ function and $\widetilde X=hX$.  Then
		$\widetilde X$ has the same equilibrium $p$ with the same kernel
		direction, and
		\[
		\widetilde a=h(p)^{2}\,a,\qquad \widetilde b=h(p)\,b .
		\]
		In particular $\{a=0\}$, $\{b=0\}$ and $\operatorname{sign}(ab)$ are
		invariants of the orbit foliation.
	\end{corollary}
	
	\begin{remark}\label{rem:orientation}
		The individual signs of $a$ and $b$ are \emph{not} invariants unless an
		orientation of the kernel line is fixed.  By Lemma~\ref{lem:scaling},
		$q_{0}\mapsto cq_{0}$ sends $(a,b)\mapsto(ca,cb)$, so for $c<0$ both
		signs reverse; only $\{a=0\}$, $\{b=0\}$ and $\operatorname{sign}(ab)$
		survive.  Once a generator $q_{0}$ has been chosen and transported, the
		individual signs are preserved by every positive orbital multiplier
		$h$, because $\widetilde a=h(p)^{2}a$ and $\widetilde b=h(p)b$ with
		$h(p)>0$.  Accordingly, whenever a numerical value or a sign of $a$ or
		$b$ is asserted below, the normalization of $q_{0}$ is stated with it;
		see \eqref{eq:q0} for the affine-bilinear class and
		Section~\ref{sec:models} for the two models.
	\end{remark}
	
	\begin{proof}
		At $p$, $D\widetilde X(p)=h(p)DX(p)$, so $\tr D\widetilde X=h\tr DX$ and
		$\dt D\widetilde X=h^{2}\dt DX$ hold at $p$; at nearby points these
		identities acquire corrections proportional to $X$.  Differentiating
		along $q_{0}$ and using $X(p)=0$, $\tr DX(p)=\dt DX(p)=0$ and
		$DX(p)q_{0}=0$ leaves
		$d(\tr D\widetilde X)(q_{0})=h(p)\,d(\tr DX)(q_{0})$ and
		$d(\dt D\widetilde X)(q_{0})=h(p)^{2}\,d(\dt DX)(q_{0})$.  Apply
		Theorem~\ref{thm:A}.
	\end{proof}
	
	Corollary~\ref{cor:orbital} is what licenses clearing a \emph{common}
	denominator in a rational model before computing $a$ and $b$; it is not
	apparent from \eqref{eq:abdef}.
	
	\begin{proposition}[Product structure]\label{prop:product}
		Let $h$ and $Z$ be of class $C^{2}$ near $p$ and let $X=h\,Z$ with
		\[
		h(p)=0,\qquad \nabla h(p)\neq0,\qquad Z(p)\neq0 .
		\]
		Then $X(p)=0$, the matrix $J:=DX(p)=Z(p)\,\nabla h(p)^{\top}$ has rank
		one, and $\ker J=T_{p}\{h=0\}$.  Moreover $J$ is nilpotent if and only
		if $\ip{\nabla h(p)}{Z(p)}=0$, and in that case, for $q_{0}$ spanning
		$T_{p}\{h=0\}$,
		\begin{equation}\label{eq:abproduct}
			a=0,
			\qquad
			b=\ip{Z(p)}{D^{2}h(p)\,q_{0}}
			+\ip{\nabla h(p)}{DZ(p)\,q_{0}} .
		\end{equation}
		The same conclusion $a=0$ holds, more generally, whenever a $C^{1}$
		curve of equilibria of $X$ passes through $p$ and $\rank DX(p)=1$.
	\end{proposition}
	
	\begin{proof}
		From $DX=h\,DZ+Z\,\nabla h^{\top}$ and $h(p)=0$ we get
		$J=Z(p)\nabla h(p)^{\top}$, of rank one because both factors are
		nonzero, with kernel $\ker\nabla h(p)^{\top}$ and trace
		$\ip{\nabla h(p)}{Z(p)}$; its determinant vanishes, so it is nilpotent
		exactly when its trace does.  The same computation at any $w$ with
		$h(w)=0$ gives $DX(w)=Z(w)\nabla h(w)^{\top}$, of rank at most one, so
		$\dt DX\equiv0$ on the curve $\{h=0\}$; as $q_{0}$ is tangent to that
		curve, $\ip{\nabla\!\dt DX(p)}{q_{0}}=0$ and Theorem~\ref{thm:A} gives
		$a=0$.  Finally $\tr DX=h\,\tr DZ+\ip{\nabla h}{Z}$; differentiating
		along $q_{0}$ and using $h(p)=0$, $\ip{\nabla h(p)}{q_{0}}=0$ leaves
		the two terms in \eqref{eq:abproduct}.
		
		For the last assertion, let $\Gamma$ be a curve of equilibria through
		$p$.  Differentiating $X|_{\Gamma}\equiv0$ gives
		$T_{p}\Gamma\subset\ker DX(p)$, and both are one-dimensional; at every
		point of $\Gamma$ the Jacobian has the tangent to $\Gamma$ in its
		kernel, so $\dt DX\equiv0$ along $\Gamma$ and Theorem~\ref{thm:A}
		applies.
	\end{proof}
	
	\begin{remark}\label{rem:computational}
		Theorem~\ref{thm:A} reduces the computation of $a$ and $b$ to two
		directional derivatives of $\tr DX$ and $\dt DX$, replacing the
		solution of four linear systems and a biorthonormalization.  All closed
		forms in Sections~\ref{sec:family}--\ref{sec:models} were obtained this
		way and cross-checked against \eqref{eq:abdef}.
	\end{remark}

	\section{The kernel line field and the two degeneracies}
	\label{sec:geometry}
	
	Let $X\colon\R^{2}\times\R^{r}\to\R^{2}$ be a smooth family, and assume
	$0$ is a regular value of $X$, so that
	$E:=\{(x,\lambda): X(x,\lambda)=0\}$ is an $r$-dimensional manifold.
	Let $\pi\colon E\to\R^{r}$ be the restriction of the projection, and on
	$E$ define
	\[
	T:=\tr D_{x}X ,\qquad
	D:=\dt D_{x}X ,\qquad
	\Sigma:=\{D=0\},\qquad
	\Hh:=\{T=0\} .
	\]
	The set $\Sigma$ is the fold, or limit-point, set of the family, and
	$\Hh$ is the neutrality set, whose two halves $\{D>0\}$ and $\{D<0\}$
	consist of Hopf points and of neutral saddles.  The Bogdanov--Takens
	set is $\Sigma\cap\Hh$.
	
	\begin{lemma}\label{lem:kernel}
		For $q=(x_{0},\lambda_{0})\in E$,
		$\ker\bigl(d\pi_{q}|_{T_{q}E}\bigr)=\ker D_{x}X(q)\times\{0\}$.  In
		particular this kernel is a line $\ell_{q}$ exactly on
		$\Sigma\setminus\{D_{x}X=0\}$, spanned by $(q_{0},0)$ with
		$q_{0}\in\ker D_{x}X(q)$.
	\end{lemma}
	
	\begin{proof}
		$T_{q}E=\{(\xi,\eta): D_{x}X\,\xi+D_{\lambda}X\,\eta=0\}$ and
		$d\pi(\xi,\eta)=\eta$.
	\end{proof}
	
	\begin{theorem}[Theorem B]\label{thm:B}
		Let $q\in E$ be a Bogdanov--Takens point, that is $T(q)=D(q)=0$ and
		$\rank D_{x}X(q)=1$, and let $\ell_{q}=\langle(q_{0},0)\rangle$.  Then,
		with $a,b$ the coefficients associated with $q_{0}$ (see
		Figure~\ref{fig:kernel}),
		\begin{enumerate}[label=\rm(\roman*)]
			\item $dD_{q}(q_{0},0)=-2a$ and $dT_{q}(q_{0},0)=b$;
			\item $a\neq0$ if and only if
			$\ell_{q}\not\subset T^{\mathrm{lin}}_{q}\Sigma
			:=\ker\bigl(dD_{q}|_{T_{q}E}\bigr)$; in that case
			$dD_{q}|_{T_{q}E}\neq0$, so $\Sigma$ is a hypersurface of $E$ near $q$
			with $T_{q}\Sigma=T^{\mathrm{lin}}_{q}\Sigma$, and $\pi|_{E}$ is a
			Whitney fold at $q$;
			\item if $dT_{q}|_{T_{q}E}\neq0$, so that $\Hh$ is a hypersurface of
			$E$ near $q$ and $T_{q}\Hh=\ker\bigl(dT_{q}|_{T_{q}E}\bigr)$, then
			\[
			b=0
			\iff \ell_{q}\subset T_{q}\Hh
			\iff \pi|_{\Hh}\ \text{is not an immersion at }q ;
			\]
			\item conversely, if $dT_{q}|_{T_{q}E}\neq0$, if $\pi|_{\Hh}$ fails to
			be an immersion at $q\in\Hh$ and if $D_{x}X(q)\neq0$, then $q$ is a
			Bogdanov--Takens point with $b=0$.
		\end{enumerate}
	\end{theorem}
	
	\begin{remark}\label{rem:regularityB}
		Two points of regularity are worth making explicit.  In~(ii) the
		linearized, or Zariski, tangent space
		$T^{\mathrm{lin}}_{q}\Sigma=\ker(dD_{q}|_{T_{q}E})$ is used rather than
		$T_{q}\Sigma$, because when $a=0$ it may happen that
		$dD_{q}|_{T_{q}E}=0$, in which case $\Sigma$ is not a regular level set
		near $q$ and $T_{q}\Sigma$ is undefined; with the linearized
		convention, $T^{\mathrm{lin}}_{q}\Sigma=T_{q}E$ in that degenerate case
		and the equivalence still reads correctly.  Parts~(iii) and~(iv) carry
		the hypothesis $dT_{q}|_{T_{q}E}\neq0$ for the same reason: $\Hh$ must
		be a smooth hypersurface before $\pi|_{\Hh}$ can be called an
		immersion.  Throughout, $\Sigma$ and $\Hh$ denote the germs at $q$ of
		the corresponding level sets, not their global loci.
	\end{remark}
	
	\begin{proof}
		(i) is Theorem~\ref{thm:A} read on $E$: the functions $T$ and $D$ are
		the restrictions to $E$ of $\tr D_{x}X$ and $\dt D_{x}X$, and
		$(q_{0},0)$ is tangent to $E$ by Lemma~\ref{lem:kernel}.
		
		(ii) If $a\neq0$ then $dD_{q}(q_{0},0)=-2a\neq0$ by~(i), so
		$dD_{q}|_{T_{q}E}\neq0$, the set $\Sigma$ is a hypersurface with
		$T_{q}\Sigma=T^{\mathrm{lin}}_{q}\Sigma=\ker(dD_{q}|_{T_{q}E})$, and
		$\ell_{q}\not\subset T_{q}\Sigma$.  Since $\Sigma$ is the critical set
		of $\pi|_{E}$, the condition $\ell_{q}\oplus T_{q}\Sigma=T_{q}E$ is the
		standard criterion for a Whitney fold.  Conversely, if $a=0$ then
		$dD_{q}(q_{0},0)=0$, so $\ell_{q}\subset\ker(dD_{q}|_{T_{q}E})
		=T^{\mathrm{lin}}_{q}\Sigma$ whether or not $dD_{q}|_{T_{q}E}$
		vanishes, and the criterion fails.
		
		(iii) $T_{q}\Hh=\ker dT_{q}$ and $dT_{q}(q_{0},0)=b$, giving the first
		equivalence; $\pi|_{\Hh}$ fails to be an immersion at $q$ exactly when
		$T_{q}\Hh$ contains a nonzero vector of $\ell_{q}$.
		
		(iv) If $\pi|_{\Hh}$ is not an immersion, $T_{q}\Hh$ meets
		$\ker d\pi_{q}|_{T_{q}E}$ nontrivially, so $\dt D_{x}X(q)=0$; together
		with $\tr D_{x}X(q)=0$ this makes $D_{x}X(q)$ nilpotent, and nonzero by
		hypothesis.  Then $\ell_{q}\subset T_{q}\Hh$ and (iii) gives $b=0$.
	\end{proof}
	
	\begin{figure}[ht]
		\centering
		\begin{tikzpicture}[line cap=round,>=latex,scale=0.87]
			\begin{scope}[xshift=0cm]
				\draw[gray!65,thin] (0,0) -- (3.6,0) -- (4.4,2.0) -- (0.8,2.0) -- cycle;
				\node[gray!65] at (0.40,1.70) {\footnotesize $E$};
				\draw (0.65,0.50) .. controls (1.50,0.72) .. (2.30,1.00)
				.. controls (3.10,1.28) .. (3.95,1.50);
				\node at (4.22,1.63) {\footnotesize $\Sigma$};
				\draw[dashed] (1.05,1.85) .. controls (1.70,1.42) .. (2.30,1.00)
				.. controls (2.90,0.58) .. (3.52,0.30);
				\node at (0.88,1.99) {\footnotesize $\Hh$};
				\draw[very thick,<->] (2.30,0.42) -- (2.30,1.58);
				\node at (2.62,1.52) {\footnotesize $\ell$};
				\fill (2.30,1.00) circle (1.6pt);
				\node at (1.90,0.65) {\footnotesize $q$};
				\node at (2.20,-0.62) {\footnotesize (a)\ \ $a\neq0,\ b\neq0$};
			\end{scope}
			\begin{scope}[xshift=4.95cm]
				\draw[gray!65,thin] (0,0) -- (3.6,0) -- (4.4,2.0) -- (0.8,2.0) -- cycle;
				\node[gray!65] at (0.40,1.70) {\footnotesize $E$};
				\draw (0.65,0.50) .. controls (1.50,0.72) .. (2.30,1.00)
				.. controls (3.10,1.28) .. (3.95,1.50);
				\node at (4.22,1.63) {\footnotesize $\Sigma$};
				\draw[dashed] (1.05,1.85) .. controls (1.70,1.42) .. (2.30,1.00)
				.. controls (2.90,0.58) .. (3.52,0.30);
				\node at (0.88,1.99) {\footnotesize $\Hh$};
				\draw[very thick,<->] (1.82,1.34) -- (2.78,0.66);
				\node at (2.98,0.75) {\footnotesize $\ell$};
				\fill (2.30,1.00) circle (1.6pt);
				\node at (1.90,0.65) {\footnotesize $q$};
				\node at (2.20,-0.62) {\footnotesize (b)\ \ $a\neq0,\ b=0$};
			\end{scope}
			\begin{scope}[xshift=9.90cm]
				\draw[gray!65,thin] (0,0) -- (3.6,0) -- (4.4,2.0) -- (0.8,2.0) -- cycle;
				\node[gray!65] at (0.40,1.70) {\footnotesize $E$};
				\draw (0.65,0.50) .. controls (1.50,0.72) .. (2.30,1.00)
				.. controls (3.10,1.28) .. (3.95,1.50);
				\node at (4.22,1.63) {\footnotesize $\Sigma$};
				\draw[dashed] (1.05,1.85) .. controls (1.70,1.42) .. (2.30,1.00)
				.. controls (2.90,0.58) .. (3.52,0.30);
				\node at (0.88,1.99) {\footnotesize $\Hh$};
				\draw[very thick,<->] (1.74,0.81) -- (2.86,1.19);
				\node at (3.02,1.45) {\footnotesize $\ell$};
				\fill (2.30,1.00) circle (1.6pt);
				\node at (1.90,0.60) {\footnotesize $q$};
				\node at (2.20,-0.62) {\footnotesize (c)\ \ $a=0,\ b\neq0$};
			\end{scope}
		\end{tikzpicture}
		\caption{The three configurations at a Bogdanov--Takens point
			$q\in\Sigma\cap\Hh$.  The quadrilateral is a patch of the equilibrium
			manifold $E$; the solid curve is the fold set $\Sigma=\{\dt J=0\}$, the
			dashed curve is the neutrality set $\Hh=\{\tr J=0\}$, and the
			double arrow is the line $\ell=\ker(d\pi|_{TE})$, the direction that
			the projection onto parameter space collapses.  By
			Theorem~\ref{thm:B}, $a$ measures the failure of $\ell$ to be tangent
			to $\Sigma$ and $b$ its failure to be tangent to $\Hh$: in (a) the
			point is nondegenerate; in (b) $\ell$ is tangent to $\Hh$, so
			$\pi|_{\Hh}$ is not an immersion; in (c) $\ell$ is tangent to
			$\Sigma$, so $\pi|_{E}$ is not a Whitney fold.  Note that $\Sigma$ and
			$\Hh$ remain transverse in all three pictures
			(Remark~\ref{rem:nottransversality}).  By Theorem~\ref{thm:C},
			configuration (c) cannot occur with an isolated equilibrium in the
			class of Section~\ref{sec:family}.}
		\label{fig:kernel}
	\end{figure}
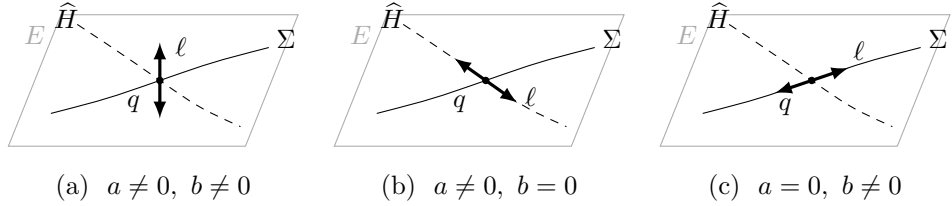
	
	\begin{remark}\label{rem:nottransversality}
		It is tempting to read $b=0$ as a loss of transversality between
		$\Sigma$ and $\Hh$ inside $E$.  This is false as soon as $p\geq3$.  For
		Model~2 of Section~\ref{sec:models}, at the degenerate point
		$q=(x,y,\alpha,\beta,\gamma)=(4,4,\tfrac14,\tfrac14,1)$, where $b=0$,
		one computes
		\[
		dT|_{q}=\bigl(\tfrac3{16},-\tfrac3{16},1,-1,-1\bigr),
		\qquad
		dD|_{q}=\bigl(-\tfrac1{16},-\tfrac1{16},-3,-3,-\tfrac12\bigr),
		\]
		while the vectors $w_{1}=(1,1,0,0,0)$ and $w_{2}=(0,0,1,-1,0)$ both
		lie in the three-dimensional space $T_{q}E$.  Since
		\[
		\det\begin{pmatrix}
			dT(w_{1})&dT(w_{2})\\[2pt] dD(w_{1})&dD(w_{2})
		\end{pmatrix}
		=\det\begin{pmatrix}0&2\\[2pt] -\tfrac18&0\end{pmatrix}
		=\tfrac14\neq0 ,
		\]
		the restrictions of $dT$ and $dD$ to $T_{q}E$ are linearly
		independent, and no choice of basis is involved: $\Sigma$ and $\Hh$
		\emph{are} transverse.  This must be so, since $\Sigma\cap\Hh$ is a curve, of the
		expected dimension $p-2=1$.  What degenerates is the position of the
		kernel line $\ell_{q}$ relative to $\Hh$, not the relative position of
		$\Sigma$ and $\Hh$.
	\end{remark}
	
	\begin{remark}\label{rem:thomboardman}
		Theorem~\ref{thm:B}(ii) places a point with $a=0$ in the
		Thom--Boardman stratum $\Sigma^{1,1}$ of $\pi|_{E}$.  It does not by
		itself make the point a cusp: being a stable cusp requires a further
		nondegeneracy, which by Theorem~\ref{thm:C} is never available in the
		class of Section~\ref{sec:family}.
	\end{remark}
	
	\begin{remark}\label{rem:notextremum}
		Theorem~\ref{thm:B} is stated in terms of the line field $\ell$ and not
		in terms of a critical value of a distinguished parameter, because the
		latter formulation is not invariant.  In Model~2 of
		Section~\ref{sec:models} the parameter $\alpha$ restricted to the BT
		curve is $\gamma/(1+\gamma)^{2}$, whose derivative
		$(1-\gamma)/(1+\gamma)^{3}$ vanishes at the degenerate point
		$\gamma=1$; but the linear change of parameter coordinates
		$\alpha'=\alpha+\beta$ gives $\alpha'=1/(1+\gamma)$ along the same
		curve, with nowhere vanishing derivative.  The degeneracy $b=0$ at
		$\gamma=1$ is of course unaffected.
	\end{remark}

	\section{Multiplicity and the mixed-volume bound}
	\label{sec:mult}
	
	Unlike Theorem~\ref{thm:A}, which needs only the $2$-jet and holds for
	$C^{2}$ fields, everything in this section is an analytic, formal or
	algebraic statement: the local multiplicity and the Takens form are not
	claimed for merely $C^{2}$ vector fields.
	
	
	\subsection{Multiplicity of a nilpotent equilibrium}
	
	For an isolated zero $p$ of a real analytic planar vector field $X$ we
	write $\mu_{p}(X)=\dim_{\C}\C[[x-x_{0},y-y_{0}]]/(X_{1},X_{2})$, the
	local ring being taken over $\C$ with coordinates centred at
	$p=(x_{0},y_{0})$ even when $X$ has real coefficients, and similarly
	$\mu_{p}(A,B)$ for a pair of functions.
	
	\begin{lemma}[Uniqueness of the quadratic part]\label{lem:quadratic}
		Let $X$ be a formal planar vector field with $X(0)=0$ and
		$DX(0)=N:=\left(\begin{smallmatrix}0&1\\0&0\end{smallmatrix}\right)$.
		Among the formal changes of coordinates tangent to the identity, the
		quadratic part of $X$ can be brought to
		$a\,u^{2}\partial_{v}+b\,uv\,\partial_{v}$, and the pair $(a,b)$ so
		obtained depends only on the $2$-jet of $X$.
	\end{lemma}
	
	\begin{proof}
		Let $H_{2}$ be the six-dimensional space of homogeneous quadratic
		vector fields and $L(h):=Dh\cdot Nz-Nh$ the homological operator, so
		that the substitution $z\mapsto z+h(z)$ with $h\in H_{2}$ replaces the
		quadratic part $X_{2}$ by $X_{2}-L(h)$ and leaves the linear part
		unchanged.  Writing $h=(h_{1},h_{2})$,
		\[
		L(h)=\bigl(v\,\partial_{u}h_{1}-h_{2},\;v\,\partial_{u}h_{2}\bigr).
		\]
		If $L(h)=0$ then $\partial_{u}h_{2}=0$, so $h_{2}=cv^{2}$, and then
		$v\,\partial_{u}h_{1}=cv^{2}$ gives $h_{1}=cuv+c'v^{2}$; hence
		$\dim\ker L=2$ and $\dim\operatorname{Im}L=4$.  The second component
		of any element of $\operatorname{Im}L$ lies in
		$\operatorname{span}\{uv,v^{2}\}$, which excludes
		$u^{2}\partial_{v}$; and an element of $\operatorname{Im}L$ whose
		first component vanishes has $c_{5}=0$ in the notation above, so its
		second component lies in $\operatorname{span}\{v^{2}\}$, which
		excludes $uv\,\partial_{v}$.  The same two observations show that no
		nonzero combination $\alpha u^{2}\partial_{v}+\beta uv\,\partial_{v}$
		lies in $\operatorname{Im}L$.  Thus
		$\operatorname{span}\{u^{2}\partial_{v},uv\,\partial_{v}\}$ is a
		complement of $\operatorname{Im}L$ in $H_{2}$, and $X_{2}$ has a unique
		representative in it.  Changes of order $\ge3$ do not affect the
		quadratic part.
	\end{proof}
	
	\begin{lemma}\label{lem:mult}
		Let $X$ be real analytic with an isolated equilibrium $p$ at which
		$DX(p)$ is nilpotent and nonzero.  Then, in the Takens form
		$\dot u=v$, $\dot v=f(u)+v\,g(u)$,
		\[
		\mu_{p}(X)=\operatorname{ord}f\geq2 ,
		\]
		and $\mu_{p}(X)=2$ if and only if $a\neq0$.
	\end{lemma}
	
	\begin{proof}
		A formal change of coordinates brings the germ to the Takens form
		\cite{Takens1974,Dumortier1977,DLA2006}, and $\mu_{p}$ is unchanged,
		being the dimension of the local algebra of the formal ideal.  Then
		\[
		\mu_{p}(X)=\dim_{\C}\C[[u,v]]/(v,\,f(u)+v\,g(u))
		=\dim_{\C}\C[[u]]/(f)=\operatorname{ord}f ,
		\]
		finite because $p$ is isolated; we set $\ordop 0=\infty$, so that the
		lemma is stated only for isolated zeros, where $f\not\equiv0$.  The
		inequality
		$\operatorname{ord}f\ge2$ holds because in the Takens form
		\[
		DX(0)=\begin{pmatrix}0&1\\ f'(0)&g(0)\end{pmatrix},
		\]
		so nilpotency forces $g(0)=\tr DX(0)=0$ and $f'(0)=-\dt DX(0)=0$,
		while $f(0)=0$ because $p$ is an equilibrium.
		
		It remains to identify the coefficient of $u^{2}$ in $f$ with the $a$
		of \eqref{eq:abdef}.  Both \eqref{eq:BT} and the Takens form have
		$\dot u=v$ and quadratic part contained in
		$\operatorname{span}\{u^{2}\partial_{v},uv\,\partial_{v}\}$: for the
		Takens form, $f(u)+v\,g(u)$ contributes $a_{2}u^{2}+b_{1}uv$ at order
		two.  By Lemma~\ref{lem:quadratic} that representative is unique, so
		$a_{2}$ is the $a$ of \eqref{eq:BT}, which the proof of
		Theorem~\ref{thm:A} identifies with \eqref{eq:abdef} computed from a
		basis adapted to $q_{0}$; by Lemma~\ref{lem:scaling} a different
		normalization of $q_{0}$ multiplies it by a nonzero constant.
		Consequently $\operatorname{ord}f=2$ if and only if $a\neq0$.
	\end{proof}
	
	\subsection{The mixed-volume bound}
	
	For a Laurent polynomial $A\in\C[x^{\pm1},y^{\pm1}]$ let $\Newt A$
	denote the convex hull of its support, and for lattice polytopes
	$\Delta,\Delta'$ let
	\[
	\MV(\Delta,\Delta')
	=\operatorname{Area}(\Delta+\Delta')-\operatorname{Area}(\Delta)
	-\operatorname{Area}(\Delta')
	\]
	be the mixed volume, normalized so that Bernstein's bound is $\MV$
	itself.  We use three classical facts: $\MV$ is symmetric, nonnegative
	and additive in each argument with respect to Minkowski sums, and it is
	monotone, so that $\MV(\Delta,\Delta')\le\MV(\Theta,\Delta')$ whenever
	$\Delta\subseteq\Theta$ \cite[Ch.~5]{Schneider2014}; and
	$\Newt(AA')=\Newt A+\Newt A'$ \cite[Ch.~7]{CLO}.
	
	\begin{lemma}\label{lem:BKK}
		Let $A,B$ be nonzero Laurent polynomials in two variables and let
		$p\in(\C^{*})^{2}$.
		\begin{enumerate}[label=\rm(\roman*)]
			\item If $A$ and $B$ have no common nonconstant factor, then
			$V(A)\cap V(B)\cap(\C^{*})^{2}$ is finite and
			$\sum_{z}\mu_{z}(A,B)\le\MV(\Newt A,\Newt B)$.
			\item If $p$ is a common zero of $A$ and $B$ and their greatest common
			divisor $h$ satisfies $h(p)\neq0$, then $p$ is isolated in
			$V(A)\cap V(B)$ and $\mu_{p}(A,B)\le\MV(\Newt A,\Newt B)$.
		\end{enumerate}
	\end{lemma}
	
	\begin{proof}
		(i) By the theorem of Bernstein and Kushnirenko
		\cite{Bernstein1975,Kushnirenko1976}, see also \cite[Sect.~7.5]{CLO}, a
		system with the given supports and generic coefficients has exactly
		$\MV(\Newt A,\Newt B)$ solutions in $(\C^{*})^{2}$, all simple.
		Perturbing the coefficients of $A$ and $B$ within their supports, the
		Newton polytopes are unchanged and, by conservation of the local
		intersection multiplicity under deformation
		\cite[Ch.~12]{Fulton1998}, every isolated zero $z$ of $(A,B)$ in the
		torus splits into
		exactly $\mu_{z}(A,B)$ zeros of the perturbed system in a small
		neighbourhood of $z$.  Since $A$ and $B$ have no common factor, all
		zeros are isolated and finite in number, and the bound follows.
		
		(ii) Write $A=hA'$, $B=hB'$ with $A',B'$ coprime.  Near $p$ the
		function $h$ is a unit, so $\mu_{p}(A,B)=\mu_{p}(A',B')$ and $p$ is
		isolated by (i).  By $\Newt A=\Newt h+\Newt A'$ and the additivity and
		nonnegativity of the mixed volume,
		\begin{multline*}
			\MV(\Newt A,\Newt B)
			=\MV(\Newt h,\Newt h)+\MV(\Newt h,\Newt B')\\
			+\MV(\Newt A',\Newt h)+\MV(\Newt A',\Newt B')
			\ \geq\ \MV(\Newt A',\Newt B') ,
		\end{multline*}
		and (i) applied to $(A',B')$ concludes.
	\end{proof}
	
	\begin{remark}\label{rem:square}
		Only solutions in the torus are counted, so solutions on the axes and
		at infinity impose no constraint.  For the case used in
		Section~\ref{sec:family}, namely $\Newt A=\Newt B=[0,1]^{2}$, one gets
		$\MV=\operatorname{Area}([0,2]^{2})-2=2$, a value that can also be
		obtained by elementary means: the projective closures of two
		affine-bilinear curves $\alpha_{1}+\alpha_{2}x+\alpha_{3}y+\alpha_{4}xy$
		are conics through $[1:0:0]$ and $[0:1:0]$, so B\'ezout's four
		intersection points leave at most two in the affine plane
		\cite[Ch.~5]{Fulton}.
	\end{remark}
	
	\subsection{Theorem C}
	
	\begin{theorem}[Theorem C]\label{thm:C}
		Let $A,B,g_{1},g_{2}$ be Laurent polynomials in two variables and let
		\[
		X=\Bigl(x\,\frac{A}{g_{1}},\;y\,\frac{B}{g_{2}}\Bigr).
		\]
		Let $p\in(\C^{*})^{2}\cap\R^{2}$ satisfy $g_{1}(p)g_{2}(p)\neq0$,
		$X(p)=0$, and let $DX(p)$ be nilpotent and nonzero.  Then:
		\begin{enumerate}[label=\rm(\roman*)]
			\item if $A$ and $B$ have a common nonconstant factor vanishing at $p$,
			then $p$ is not an isolated equilibrium and $a=0$;
			\item otherwise $p$ is isolated and
			\[
			\operatorname{ord}f=\mu_{p}(X)=\mu_{p}(A,B)
			\le\MV(\Newt A,\Newt B).
			\]
		\end{enumerate}
	\end{theorem}
	
	\begin{proof}
		Since $g_{1}(p)g_{2}(p)\neq0$ and $p$ lies in the torus, the functions
		$x,y,g_{1},g_{2}$ are units in the local ring at $p$, whence
		$(X_{1},X_{2})_{p}=(A,B)_{p}$ and $\mu_{p}(X)=\mu_{p}(A,B)$.  Note that
		this uses only that the $g_{i}$ are units at $p$; no common factor is
		required, and the auxiliary system $(A,B)$ need not be orbitally
		equivalent to $X$.
		
		(i) Let $h$ be such a factor and write $A=h\widetilde A$,
		$B=h\widetilde B$, so that
		$X=h\,Z$ with $Z=(x\widetilde A/g_{1},\,y\widetilde B/g_{2})$.  Since
		$DX(p)=Z(p)\,\nabla h(p)^{\top}$ has rank one, both $Z(p)$ and
		$\nabla h(p)$ are nonzero; Proposition~\ref{prop:product} gives $a=0$,
		and the implicit function theorem makes $\{h=0\}$ a curve of equilibria
		through $p$.
		
		(ii) By Lemma~\ref{lem:BKK}, $p$ is isolated and
		$\mu_{p}(A,B)\le\MV(\Newt A,\Newt B)$; Lemma~\ref{lem:mult} identifies
		$\mu_{p}(X)$ with $\operatorname{ord}f$.
	\end{proof}
	
	\begin{corollary}\label{cor:cusp}
		Assume the hypotheses of Theorem~\ref{thm:C} and write
		$m=\operatorname{ord}f$, $n=\operatorname{ord}g$ for the orders of the
		Takens form at $p$.
		\begin{enumerate}[label=\rm(\roman*)]
			\item If $\MV(\Newt A,\Newt B)\le1$, then $p$ is not isolated:
			equivalently, $A$ and $B$ have a common nonconstant factor vanishing at
			$p$, and $a=0$, so that $p$ lies on a curve of equilibria.
			Contrapositively, an isolated equilibrium in $(\C^{*})^{2}$ with
			nilpotent nonzero Jacobian forces $\MV(\Newt A,\Newt B)\ge2$.  The
			qualification is essential: see Remark~\ref{rem:notvacuous}.
			\item If $\MV(\Newt A,\Newt B)=2$, then either $p$ is not isolated, or
			$m=2$ and $a\neq0$; in the latter case the germ of $X$ at $p$ is
			topologically a cusp, for every admissible value of $n$.
		\end{enumerate}
		Consequently, if $\MV(\Newt A,\Newt B)\le2$ then no isolated interior
		equilibrium can be a nilpotent singularity of saddle, focus or elliptic
		type: all of these have $m\ge3$, and they are the codimension-three
		cases unfolded in \cite{DRSZ1991}.  The bound constrains $m$ alone and
		imposes no restriction on $n$, so the vanishing of $b$, and more
		generally degeneracies of higher order in $g$, remain compatible with
		$\MV\le2$; both models of Section~\ref{sec:models} satisfy
		$\MV\le2$ and possess a point at which $b=0$.
	\end{corollary}
	
	\begin{proof}
		At an isolated equilibrium with nilpotent nonzero Jacobian
		Lemma~\ref{lem:mult} gives $m=\mu_{p}(X)\ge2$, while
		Theorem~\ref{thm:C}(ii) gives $m\le\MV(\Newt A,\Newt B)$.
		
		For~(i), suppose $p$ were isolated.  Then Theorem~\ref{thm:C}(ii)
		applies and yields $m\le\MV\le1$, contradicting $m\ge2$.  Hence the
		alternative of Theorem~\ref{thm:C}(i) holds: $A$ and $B$ share a
		nonconstant factor vanishing at $p$, the point is not isolated and
		$a=0$.  For~(ii), if $\MV=2$ the two inequalities force $m=2$, hence
		$a\neq0$ by Lemma~\ref{lem:mult}.
		
		For the topological type, recall that a nilpotent singularity with $m$
		even and $2n\ge m$ is a cusp \cite[Ch.~3]{DLA2006}, \cite{Dumortier1977}.
		Nilpotency of $DX(p)$ forces $g(0)=0$, that is $n\ge1$, so for $m=2$
		the inequality $2n\ge m$ holds for every admissible $n$.  Equivalently,
		in the quasi-homogeneous weights determined by $\dot u=v$ and by
		$a_{2}u^{2}$, the term $b_{n}u^{n}v$ has strictly larger weight for
		every $n\ge1$, and the principal part is the Hamiltonian field
		$\dot u=v$, $\dot v=a_{2}u^{2}$, whose origin is a cusp.  The saddle,
		focus and elliptic types occur only for $m\ge3$, which gives the last
		assertion.
	\end{proof}
	
	\begin{remark}\label{rem:notvacuous}
		The qualification in Corollary~\ref{cor:cusp}(i) is not vacuous: the
		excluded points may exist, and only their isolation fails.  Take
		\[
		A=(x-1)(y-1),\qquad B=x-1,\qquad g_{1}=g_{2}=1,
		\]
		and $X=(xA,yB)$.  At $p=(1,1)$,
		\[
		DX(p)=\begin{pmatrix}0&0\\1&0\end{pmatrix},
		\]
		which is nilpotent of rank one and nonzero, while $\Newt A$ is the unit
		square, $\Newt B$ the unit segment, and
		$\MV(\Newt A,\Newt B)=1$.  The point is interior and does occur; what
		fails is isolation, the whole line $\{x=1\}$ consisting of equilibria,
		in accordance with the common factor $x-1$ of $A$ and $B$.  Consistently
		with Theorem~\ref{thm:A}, here $q_{0}=(0,1)$, $a=0$ and $b=1$.
	\end{remark}
	
	\begin{remark}\label{rem:interior}
		Theorem~\ref{thm:C} and Corollary~\ref{cor:cusp} are statements about
		interior equilibria, and this is the natural range of the argument
		rather than a restriction adopted for convenience.  Two independent
		reasons place the coordinate axes outside it.  First, for a Kolmogorov
		system the axes are invariant for every value of the parameters, so a
		nilpotent singularity lying on an axis must be unfolded within the
		class of vector fields preserving that line, and the classification
		invoked in Corollary~\ref{cor:cusp} is not the relevant one there.
		Second, the multiplicity argument itself ceases to apply: the
		identification $\mu_{p}(X)=\mu_{p}(A,B)$ used in the proof of
		Theorem~\ref{thm:C} requires $x$ and $y$ to be units in the local ring
		at $p$, which fails as soon as one coordinate vanishes; at a point
		$(x_{0},0)$ with $x_{0}\neq0$ one would have to control
		$\dim_{\C}\C[[x,y]]/(A,\,yB)$ instead.  Consistently with this, the
		bound of Lemma~\ref{lem:BKK} is blind to the axes by construction,
		since Bernstein's theorem counts solutions lying in $(\C^{*})^{2}$
		only, and extending it to the affine plane calls for a different
		combinatorial invariant.  Every statement of this section therefore
		concerns equilibria with $x_{0}y_{0}\neq0$.
	\end{remark}

	\section{The class with cross-product cubic terms}
	\label{sec:family}
	
	\subsection{The class and the reduction of rational models}
	
	We consider the planar polynomial systems
	\begin{equation}\label{eq:family}
		\dot x = x\,P(x,y),\qquad
		\dot y = y\,Q(x,y),
	\end{equation}
	with
	\begin{equation}\label{eq:PQ}
		P = -\alpha_{1}+\alpha_{2}x+\alpha_{3}y+\alpha_{4}xy,
		\qquad
		Q = -\beta_{1}+\beta_{3}x+\beta_{2}y+\beta_{4}xy .
	\end{equation}
	Expanded, \eqref{eq:family} is a cubic system whose cubic terms are
	exclusively the cross-products $x^{2}y$ and $xy^{2}$.  Since
	$\Newt P,\Newt Q\subseteq[0,1]^{2}$, so by monotonicity of the mixed
	volume and Remark~\ref{rem:square},
	$\MV(\Newt P,\Newt Q)\le2$; equality of the polytopes with the full
	square requires all four corner coefficients of each polynomial to be
	nonzero, which is not assumed.  If in addition the nilpotent
	equilibrium is isolated, Corollary~\ref{cor:cusp}(i) forces
	$\MV(\Newt P,\Newt Q)\ge2$ and hence $\MV(\Newt P,\Newt Q)=2$; the
	equality is a consequence of isolation, not of membership in the class.
	In either case Theorem~\ref{thm:C} applies.
	
	\begin{remark}\label{rem:origin}
		Predator--prey models with rational interaction terms reduce to this
		setting, but two different reductions are involved and they must not be
		confused.
		
		Let $\dot x=x\widetilde P$, $\dot y=y\widetilde Q$ with
		$\widetilde P=A/g_{1}$ and $\widetilde Q=B/g_{2}$ rational, and let $p$
		be an equilibrium in the open quadrant with $g_{1}(p)g_{2}(p)\neq0$.
		\begin{enumerate}[label=\rm(\alph*)]
			\item \emph{Separate clearing.}  For the local multiplicity, and hence
			for Theorem~\ref{thm:C}, each denominator may be cleared
			independently, because $g_{1}$ and $g_{2}$ are units at $p$; the
			relevant Newton polytopes are those of $A$ and $B$.  The auxiliary
			system $(xA,yB)$ is \emph{not} orbitally equivalent to the original one
			when $g_{1}\neq g_{2}$, and must not be used to compute $a$ or $b$.
			\item \emph{Common clearing.}  For the coefficients $a$ and $b$
			themselves one needs a common factor.  Multiplying the field by
			$g=g_{1}=g_{2}$ amounts to the time rescaling
			\[
			\frac{dt}{d\tau}=g(x,y),
			\qquad\text{equivalently}\qquad
			d\tau=\frac{dt}{g(x,y)}
			\]
			along an orbit, since $dx/d\tau=(dx/dt)(dt/d\tau)=g\,X$.  By
			Corollary~\ref{cor:orbital} this replaces $(a,b)$ by
			$(g(p)^{2}a,\,g(p)b)$, so all vanishing statements are unaffected.  Only when $g\widetilde P$ and
			$g\widetilde Q$ are affine-bilinear does the system itself belong to
			\eqref{eq:family}--\eqref{eq:PQ}, and only then do the closed formulas
			of Subsection~\ref{ss:closed} apply.
		\end{enumerate}
		Holling type~II systems with satiation \cite{HollingTypeII,Bazykin1998}
		and Beddington--DeAngelis systems with mutual interference
		\cite{BeddingtonDeAngelis,DeAngelisGoldsteinONeill} satisfy (a); as
		Section~\ref{sec:models} shows, one of them satisfies (b) and the other
		does not.
	\end{remark}
	
	\subsection{Bogdanov--Takens points and the parametrization}
	
	Let $(x_{0},y_{0})$ be an equilibrium of \eqref{eq:family} with
	$x_{0}y_{0}\neq0$, so that $P(x_{0},y_{0})=Q(x_{0},y_{0})=0$ and
	\begin{equation}\label{eq:Jfactor}
		J(x_{0},y_{0})
		=\begin{pmatrix}x_{0}&0\\0&y_{0}\end{pmatrix}M,\qquad
		M:=\begin{pmatrix}P_{x}&P_{y}\\ Q_{x}&Q_{y}\end{pmatrix}(x_{0},y_{0}) .
	\end{equation}
	
	\begin{proposition}\label{prop:tangency}
		For such an equilibrium $\dt J=x_{0}y_{0}\dt M$.  Hence $\dt J=0$ if
		and only if $\nabla P$ and $\nabla Q$ are linearly dependent at
		$(x_{0},y_{0})$; when moreover $\nabla P\neq0$ and $\nabla Q\neq0$,
		this is the tangency of the nullclines $\{P=0\}$ and $\{Q=0\}$ at that
		point.  The Bogdanov--Takens conditions read
		$P=Q=0$, $\dt M=0$, $x_{0}P_{x}+y_{0}Q_{y}=0$.
	\end{proposition}
	
	Solving these for $(\alpha_{1},\beta_{1},\alpha_{4},\beta_{4})$ yields
	\begin{equation}\label{eq:param}
		\begin{aligned}
			\alpha_{1} &= \frac{-y_{0}^{2}\alpha_{3}^{2}
				+ x_{0}y_{0}\alpha_{2}(\alpha_{3}-\beta_{2})
				+ x_{0}^{2}\alpha_{2}\beta_{3}}{\Delta},
			&
			\beta_{1} &= \frac{-y_{0}^{2}\alpha_{3}\beta_{2}
				+ x_{0}y_{0}\beta_{2}(\alpha_{2}-\beta_{3})
				+ x_{0}^{2}\beta_{3}^{2}}{\Delta},
			\\[4pt]
			\alpha_{4} &= \frac{-x_{0}^{2}\alpha_{2}^{2}
				+ y_{0}^{2}\alpha_{3}\beta_{2}
				+ x_{0}y_{0}\alpha_{3}(\alpha_{2}-\beta_{3})}{x_{0}y_{0}\Delta},
			&
			\beta_{4} &= \frac{y_{0}^{2}\beta_{2}^{2}
				- x_{0}\bigl(x_{0}\alpha_{2}+y_{0}(\beta_{2}-\alpha_{3})\bigr)\beta_{3}}
			{x_{0}y_{0}\Delta},
		\end{aligned}
	\end{equation}
	where $\Delta:=-y_{0}(\alpha_{3}+\beta_{2})+x_{0}(\alpha_{2}+\beta_{3})$
	is assumed nonzero.  Setting
	\begin{equation}\label{eq:sigmatau}
		\sigma:=x_{0}\alpha_{2}-y_{0}\alpha_{3},
		\qquad
		\tau:=-y_{0}\beta_{2}+x_{0}\beta_{3},
	\end{equation}
	one has the identity
	\begin{equation}\label{eq:Delta}
		\Delta=\sigma+\tau .
	\end{equation}
	
	\begin{proposition}\label{prop:A0}
		Under \eqref{eq:param},
		\begin{equation}\label{eq:A0}
			A_{0}:=J(x_{0},y_{0})=\frac{1}{\Delta}
			\begin{pmatrix}
				\sigma\tau & -\dfrac{x_{0}\sigma^{2}}{y_{0}}\\[8pt]
				\dfrac{y_{0}\tau^{2}}{x_{0}} & -\sigma\tau
			\end{pmatrix},
		\end{equation}
		so $\tr A_{0}=\dt A_{0}=0$, and $\rank A_{0}=1$ whenever $\tau\neq0$,
		with
		\begin{equation}\label{eq:q0}
			q_{0}=\Bigl(\frac{x_{0}\sigma}{y_{0}\tau},\;1\Bigr)^{\!\top}
			\in\ker A_{0}.
		\end{equation}
		Moreover
		\begin{equation}\label{eq:alpha4}
			\alpha_{4}=-\frac{\alpha_{3}}{x_{0}}
			-\frac{\sigma^{2}}{x_{0}y_{0}(\sigma+\tau)} .
		\end{equation}
	\end{proposition}
	
	\subsection{Closed form of \texorpdfstring{$a$}{a} and
		\texorpdfstring{$b$}{b}}\label{ss:closed}
	
	By Theorem~\ref{thm:A} the coefficients are obtained from
	\eqref{eq:q0} and the two functions $\tr J$ and $\dt J$ by two
	differentiations and the substitution \eqref{eq:param}.
	
	\begin{theorem}\label{thm:ab}
		Let $(x_{0},y_{0})$ with $x_{0}y_{0}\neq0$ be a Bogdanov--Takens point
		of \eqref{eq:family}, with $\Delta\neq0$ and $\tau\neq0$.  Then, for
		the normalization \eqref{eq:q0},
		\begin{equation}\label{eq:abclosed}
			a=\frac{\sigma\,(\beta_{2}\sigma-\alpha_{3}\tau)}{\sigma+\tau},
			\qquad
			b=-\alpha_{3}-\frac{x_{0}\beta_{3}}{y_{0}\tau}\,\sigma
			=-\,\frac{y_{0}\alpha_{3}\tau+x_{0}\beta_{3}\sigma}{y_{0}\tau} .
		\end{equation}
		In particular
		\begin{equation}\label{eq:zeros}
			a=0\iff
			\sigma=0\ \ \text{or}\ \ \beta_{2}\sigma=\alpha_{3}\tau ,
			\qquad
			b=0\iff x_{0}\beta_{3}\sigma=-y_{0}\alpha_{3}\tau .
		\end{equation}
	\end{theorem}
	
	\subsection{The branches of \texorpdfstring{$\{a=0\}$}{a=0} when
		\texorpdfstring{$\tau\neq0$}{tau nonzero}}
	
	\begin{corollary}\label{cor:equiv}
		Let $(x_{0},y_{0})$ with $x_{0}y_{0}\neq0$ be an equilibrium of
		\eqref{eq:family}--\eqref{eq:PQ} at which the Jacobian is nilpotent and
		nonzero.  The following are equivalent:
		\begin{enumerate}[label=\rm(\roman*)]
			\item $a=0$;
			\item $P$ and $Q$ have a common nonconstant factor vanishing at
			$(x_{0},y_{0})$;
			\item $(x_{0},y_{0})$ is not an isolated equilibrium.
		\end{enumerate}
		No further hypothesis is required; in particular $\Delta$ and $\tau$ may
		vanish.
	\end{corollary}
	
	\begin{proof}
		Here $A=P$, $B=Q$ and $g_{1}=g_{2}=1$, and
		$\MV(\Newt P,\Newt Q)\le2$ by Remark~\ref{rem:square}.  Then (ii)
		implies (i) and (iii) by Theorem~\ref{thm:C}(i), and if (ii) fails then
		Theorem~\ref{thm:C}(ii) gives $\mu\le2$, hence $a\neq0$ by
		Lemma~\ref{lem:mult} and $(x_{0},y_{0})$ isolated.
	\end{proof}
	
	\begin{proposition}\label{prop:branches}
		Assume in addition $\Delta\neq0$, $\tau\neq0$.  The locus $\{a=0\}$
		consists, \emph{within the region $\tau\neq0$}, of exactly two
		branches, and on each of them the common factor
		of Corollary~\ref{cor:equiv} is explicit:
		\begin{enumerate}[label=\rm(\roman*)]
			\item if $\sigma=0$ then \eqref{eq:param} reduces to
			$\alpha_{1}=y_{0}\alpha_{3}$, $\beta_{1}=x_{0}\beta_{3}$,
			$\alpha_{4}=-\alpha_{3}/x_{0}$, $\beta_{4}=-\beta_{2}/x_{0}$, and
			\begin{equation}\label{eq:factor1}
				P=\frac{\alpha_{3}}{x_{0}}\,(x-x_{0})(y_{0}-y),
				\qquad
				Q=\frac{1}{x_{0}}\,(x-x_{0})(x_{0}\beta_{3}-\beta_{2}y),
			\end{equation}
			so the whole line $\{x=x_{0}\}$ consists of equilibria;
			\item if $\alpha_{2}\beta_{2}=\alpha_{3}\beta_{3}$, equivalently
			$\beta_{2}\sigma=\alpha_{3}\tau$, then $P$ and $Q$ are proportional and
			the whole conic $\{Q=0\}$ consists of equilibria.  When $\beta_{2}\neq0$
			the proportionality reads
			$\alpha_{1}/\beta_{1}=\alpha_{2}/\beta_{3}
			=\alpha_{3}/\beta_{2}=\alpha_{4}/\beta_{4}$, that is
			$P=(\alpha_{3}/\beta_{2})\,Q$; when $\beta_{2}=0$ the same conclusion
			holds with the ratio written as $P=(\alpha_{2}/\beta_{3})\,Q$, which is
			legitimate because $\tau\neq0$ and $\beta_{2}=0$ force
			$\beta_{3}\neq0$.
		\end{enumerate}
	\end{proposition}
	
	\begin{proof}
		Direct substitution of the stated values into \eqref{eq:param} and
		\eqref{eq:PQ}; the branches are those of \eqref{eq:zeros}.  For~(ii)
		with $\beta_{2}=0$, the condition $\beta_{2}\sigma=\alpha_{3}\tau$ and
		$\tau\neq0$ give $\alpha_{3}=0$, whence $\sigma=x_{0}\alpha_{2}$ and
		$\tau=x_{0}\beta_{3}$, and \eqref{eq:param} yields
		$\alpha_{i}=(\alpha_{2}/\beta_{3})\beta_{i}$ for $i=1,4$ as well.
	\end{proof}
	
	\begin{remark}\label{rem:beta2zero}
		The two branches of Proposition~\ref{prop:branches} are \emph{not}
		distinguished by the vanishing of $\sigma$ alone.  It is tempting to
		argue that when $\beta_{2}=0$ the second alternative in
		\eqref{eq:zeros} reduces to $\alpha_{3}=0$, and that then
		$a=\beta_{2}\sigma^{2}/\Delta=0$ forces $\sigma=0$; the implication is
		false, since $\beta_{2}=0$ makes that expression vanish for
		\emph{every} $\sigma$.  Concretely, take
		$x_{0}=y_{0}=1$, $\alpha_{2}=\beta_{3}=1$,
		$\alpha_{3}=\beta_{2}=0$.  Then \eqref{eq:param} gives
		$\alpha_{1}=\beta_{1}=\tfrac12$,
		$\alpha_{4}=\beta_{4}=-\tfrac12$, so that
		\[
		P=Q=-\tfrac12+x-\tfrac12\,xy ,
		\]
		the equilibrium is non-isolated and $a=0$, while
		$\sigma=1\neq0$.  This is the proportional-nullcline branch~(ii), not
		branch~(i).  This is why~(ii) is stated through
		$\alpha_{2}\beta_{2}=\alpha_{3}\beta_{3}$, without dividing by
		$\beta_{2}$.
	\end{remark}
	
	\begin{remark}\label{rem:whatisatsigma0}
		The germ at a point with $\sigma=0$ deserves a comment, because its
		normal-form coefficients are perfectly well defined and might be
		misread as a codimension-three degeneracy.  By \eqref{eq:factor1} the
		field is $X=(x-x_{0})\,Y$ with $Y$ quadratic, the Kuznetsov basis is
		finite, and $a=0$, $b=-\alpha_{3}$.  Nevertheless, in the Takens form
		the whole line of equilibria is the $u$-axis, so $f\equiv0$: all the
		coefficients of $f$ vanish, not only the quadratic one.  The
		singularity is not isolated and belongs to no finite-codimension
		stratum of the nilpotent classification.  The same happens on the
		branch $\alpha_{2}\beta_{2}=\alpha_{3}\beta_{3}$.
	\end{remark}
	
	\begin{remark}\label{rem:tauzero}
		The normalization \eqref{eq:q0} and the closed formulas
		\eqref{eq:abclosed} presuppose $\tau\neq0$, and so does
		Proposition~\ref{prop:branches}.  The excluded locus is not empty, and
		by the symmetry $x\leftrightarrow y$ of
		\eqref{eq:family}--\eqref{eq:PQ} it is the mirror image of
		branch~(i): where $\sigma=0$ produces the vertical line
		$\{x=x_{0}\}$ of equilibria, $\tau=0$ produces the horizontal line
		$\{y=y_{0}\}$.  For instance, with $x_{0}=y_{0}=1$ and
		\[
		P=x(1-y),\qquad Q=(x-1)(1-y),
		\]
		one has $\alpha_{2}=1$, $\alpha_{3}=0$, $\beta_{2}=\beta_{3}=1$,
		hence $\sigma=1$, $\tau=0$, $\Delta=1$, and
		\[
		DX(1,1)=\begin{pmatrix}0&-1\\ 0&0\end{pmatrix},
		\]
		nilpotent of rank one.  The common factor is $y-1$ and the whole line
		$\{y=1\}$ consists of equilibria, so $a=0$ by
		Corollary~\ref{cor:equiv}, as Theorem~\ref{thm:A} confirms with
		$q_{0}=(1,0)^{\top}$.  Substituting into \eqref{eq:abclosed} would
		instead return $a=1$: the formula is simply not applicable when
		$\tau=0$.  Corollary~\ref{cor:equiv} and the mixed-volume obstruction
		are unaffected, since neither uses \eqref{eq:q0}.
	\end{remark}
	
	\begin{remark}\label{rem:deltazero}
		The explicit parametrization \eqref{eq:param} is derived under
		$\Delta\neq0$.  The configurations with $\Delta=0$ carrying a nonzero
		nilpotent Jacobian are non-isolated common-factor configurations, hence
		covered by Corollary~\ref{cor:equiv}; they play no role in the
		mixed-volume obstruction, whose proof never uses \eqref{eq:param}.
	\end{remark}
	
	\begin{proposition}\label{prop:classification}
		On the Bogdanov--Takens set of \eqref{eq:family}--\eqref{eq:PQ}, under
		\begin{equation}\label{eq:noncollision}
			x_{0}y_{0}\Delta\tau\beta_{2}\beta_{3}\neq0,
			\qquad
			\alpha_{3}\neq0,
			\qquad
			y_{0}\beta_{2}+x_{0}\beta_{3}\neq0,
		\end{equation}
		the locus $ab=0$ meets each one-dimensional slice in which the
		remaining parameters are held fixed in exactly three distinct points,
		described in the variable $\sigma$ by
		\[
		\sigma^{(1)}=0,
		\qquad
		\sigma^{(2)}=-\frac{y_{0}\alpha_{3}\tau}{x_{0}\beta_{3}},
		\qquad
		\sigma^{(3)}=\frac{\alpha_{3}\tau}{\beta_{2}} ,
		\]
		and:
		\begin{enumerate}[label=\rm(\roman*)]
			\item at $\sigma^{(1)}$ and $\sigma^{(3)}$ one has $a=0$ and the
			equilibrium is not isolated;
			\item at $\sigma^{(2)}$ one has $b=0$ and
			\[
			a=\frac{y_{0}\alpha_{3}^{2}\,\tau\,(y_{0}\beta_{2}+x_{0}\beta_{3})}
			{x_{0}\beta_{3}\,(x_{0}\beta_{3}-y_{0}\alpha_{3})},
			\]
			which is nonzero under \eqref{eq:noncollision}, since the numerator
			carries the factors $\alpha_{3}$, $\tau$ and
			$y_{0}\beta_{2}+x_{0}\beta_{3}$ and, at $\sigma^{(2)}$,
			$\Delta=\tau(x_{0}\beta_{3}-y_{0}\alpha_{3})/(x_{0}\beta_{3})$, so
			that $\Delta\neq0$ and $\tau\neq0$ force the denominator to be
			nonzero as well; the equilibrium is
			isolated of multiplicity two, and the point is a degenerate
			Bogdanov--Takens singularity of cusp type, of codimension at least
			three \cite{DRS1987,Kuznetsov2005}.
		\end{enumerate}
	\end{proposition}
	
	\begin{proof}
		The three values are those of \eqref{eq:zeros}.  They are pairwise
		distinct precisely under \eqref{eq:noncollision}: one has
		$\sigma^{(2)}=\sigma^{(1)}$ or $\sigma^{(3)}=\sigma^{(1)}$ if and only
		if $\alpha_{3}\tau=0$, while
		\[
		\sigma^{(2)}-\sigma^{(3)}
		=-\frac{\alpha_{3}\tau\,(y_{0}\beta_{2}+x_{0}\beta_{3})}
		{x_{0}\beta_{2}\beta_{3}} ,
		\]
		which vanishes exactly when $\alpha_{3}\tau=0$ or
		$y_{0}\beta_{2}+x_{0}\beta_{3}=0$.  The rest is immediate from
		\eqref{eq:abclosed}, Corollary~\ref{cor:equiv} and
		Proposition~\ref{prop:branches}.  For~(ii), $a\neq0$ and
		Lemma~\ref{lem:mult} give $m=2$, while $b=0$ gives $n\ge2$; every such
		germ is a cusp \cite{Dumortier1977,DRS1987,DLA2006}.
	\end{proof}
	
	\begin{remark}\label{rem:noncollision}
		The last two conditions in \eqref{eq:noncollision} are not decorative.
		If $\alpha_{3}=0$ the three values collapse to $\sigma=0$, and if
		$y_{0}\beta_{2}+x_{0}\beta_{3}=0$ then $\sigma^{(2)}=\sigma^{(3)}$: the
		$b=0$ point merges with an $a=0$ point and the equilibrium ceases to be
		isolated.  The second coincidence is visible in the displayed value of
		$a$ at $\sigma^{(2)}$, whose numerator carries precisely the factor
		$y_{0}\beta_{2}+x_{0}\beta_{3}$.  Note also that the statement
		concerns the three points cut out on a one-dimensional slice, not the
		number of irreducible components of $\{ab=0\}$ in the full parameter
		space.
	\end{remark}
	
	\begin{remark}\label{rem:exactcodim}
		Proposition~\ref{prop:classification}(ii) asserts codimension at least
		three and not exactly three.  The codimension of a nilpotent
		singularity is determined by the pair $(m,n)$ with $m=\ordop f$ and
		$n=\ordop g$; here $m=2$ is furnished by Lemma~\ref{lem:mult} and
		$b=0$ gives $n\ge2$, but the equality $n=2$, which is what the cusp
		case of codimension three requires, is the nonvanishing of the
		coefficient of $u^{2}v$ in the Takens form.  That coefficient is not
		computed anywhere in this paper, and by Corollary~\ref{cor:equiv} it
		cannot be: the mixed volume bounds $m$ and imposes no restriction
		whatever on $n$.  The distinction is immaterial for the topological
		type, which is a cusp for every $n\ge1$ once $m=2$.
		
		We stress what is and is not asserted.  The bound leaves $n$
		undetermined; it does not exhibit systems in the class realizing
		arbitrarily large values of $n$.  The absence of an upper bound is not
		an existence statement, and constructing sparse Kolmogorov families
		with prescribed $n$ is a separate question, not addressed here.
	\end{remark}

	\section{Two ecological models}\label{sec:models}
	
	\subsection{Model 1: predator satiation}
	
	Consider
	\begin{equation}\label{eq:model1}
		\dot x = x(1-y)-\varepsilon x^{2},
		\qquad
		\dot y = -\gamma y+\frac{xy^{2}}{n+y},
		\qquad \varepsilon,\gamma,n>0 .
	\end{equation}
	Here $\widetilde P=1-\varepsilon x-y$ and
	$\widetilde Q=-\gamma+xy/(n+y)$, so that the separate clearing of
	Remark~\ref{rem:origin}(a) gives
	\begin{equation}\label{eq:AB1}
		A=1-\varepsilon x-y,\qquad B=-\gamma(n+y)+xy,
		\qquad g_{1}=1,\quad g_{2}=n+y .
	\end{equation}
	Both supports lie in the unit square, so
	$\MV(\Newt A,\Newt B)\le\MV([0,1]^{2},[0,1]^{2})=2$ by monotonicity.
	Moreover $A$ is linear, hence irreducible, and
	$B|_{\{A=0\}}=-\varepsilon x^{2}+(1+\gamma\varepsilon)x-\gamma(n+1)$
	is a nonzero polynomial because $\varepsilon>0$; therefore $A\nmid B$
	and $A,B$ are coprime.  By Theorem~\ref{thm:C}(ii) and
	Corollary~\ref{cor:cusp}, every interior equilibrium with nilpotent
	nonzero Jacobian has multiplicity two and $a\neq0$.
	
	Note that the common clearing of Remark~\ref{rem:origin}(b) would give
	$P=(n+y)(1-\varepsilon x-y)$, which is of degree two in $y$; the system
	\eqref{eq:model1} is therefore \emph{not} orbitally equivalent to a
	member of \eqref{eq:family}--\eqref{eq:PQ}, and the closed formulas
	\eqref{eq:abclosed} do not apply to it.  It is Theorem~\ref{thm:C},
	not Corollary~\ref{cor:equiv}, that governs this model.
	
	Writing $y_{0}=\zeta_{0}/[2(1+\zeta_{0})]$, $\zeta_{0}>0$, the
	Bogdanov--Takens set of \eqref{eq:model1} is the curve
	\begin{equation}\label{eq:BT1}
		x=\frac{(2+\zeta_{0})^{3}}{4\zeta_{0}(1+\zeta_{0})},
		\quad
		\gamma=\frac{(2+\zeta_{0})^{2}}{2\zeta_{0}(1+\zeta_{0})},
		\quad
		\varepsilon=\frac{2\zeta_{0}}{(2+\zeta_{0})^{2}},
		\quad
		n=\frac{\zeta_{0}^{2}}{4(1+\zeta_{0})} .
	\end{equation}
	Applying Theorem~\ref{thm:A} along \eqref{eq:BT1},
	\begin{equation}\label{eq:ab1}
		a(\zeta_{0})=-\frac{(\zeta_{0}+2)^{2}}{2\zeta_{0}(\zeta_{0}+1)},
		\qquad
		b(\zeta_{0})=\frac{\zeta_{0}-2}{\zeta_{0}} .
	\end{equation}
	Thus $a<0$ for all $\zeta_{0}>0$, as predicted, and $b$ has the unique
	positive zero $\zeta_{0}=2$, at which
	\[
	(x,y,\gamma,\varepsilon,n)=
	\Bigl(\tfrac83,\tfrac13,\tfrac43,\tfrac14,\tfrac13\Bigr),
	\qquad a=-\tfrac43,\qquad b=0 .
	\]
	This is a degenerate Bogdanov--Takens point of cusp type, of
	codimension at least three by Remark~\ref{rem:exactcodim}, and it is
	the only degenerate point on the curve.
	
	\subsection{Model 2: mutual interference}
	
	Consider
	\begin{equation}\label{eq:model2}
		\dot x = x-\frac{xy}{(1+\alpha x)(1+\beta y)},
		\qquad
		\dot y = -\gamma y+\frac{xy}{(1+\alpha x)(1+\beta y)} ,
	\end{equation}
	with $\alpha,\beta,\gamma>0$.  Both equations have the same denominator
	$g=(1+\alpha x)(1+\beta y)$, so here the two reductions of
	Remark~\ref{rem:origin} coincide, and
	\[
	A=g-y,\qquad B=-\gamma g+x
	\]
	are affine-bilinear: the time rescaling $dt/d\tau=g$, that is
	$d\tau=dt/g$, puts \eqref{eq:model2} inside
	\eqref{eq:family}--\eqref{eq:PQ}.  The
	Bogdanov--Takens set is
	\begin{equation}\label{eq:BT2}
		x=(1+\gamma)^{2},\quad
		y=\frac{(1+\gamma)^{2}}{\gamma},\quad
		\alpha=\frac{\gamma}{(1+\gamma)^{2}},\quad
		\beta=\frac{1}{(1+\gamma)^{2}} ,
	\end{equation}
	and Theorem~\ref{thm:A} gives
	\begin{equation}\label{eq:ab2}
		a(\gamma)=\frac{\gamma^{3}}{(1+\gamma)^{4}},
		\qquad
		b(\gamma)=-\frac{(\gamma-1)\gamma^{2}}{(1+\gamma)^{4}} .
	\end{equation}
	Again $a>0$ throughout, and $b$ vanishes only at $\gamma=1$, giving the
	degenerate point $(x,y,\gamma,\alpha,\beta)=(4,4,1,\tfrac14,\tfrac14)$
	with $a=\tfrac1{16}$, $b=0$.
	
	\subsection{Consistency with the general theory}
	
	Restricting the rescaled system to \eqref{eq:BT2} puts Model~2 in the
	form \eqref{eq:family}--\eqref{eq:PQ} with
	\[
	\alpha_{1}=-1,\quad
	\alpha_{2}=\frac{\gamma}{(1+\gamma)^{2}},\quad
	\alpha_{3}=-\frac{\gamma(\gamma+2)}{(1+\gamma)^{2}},\quad
	\alpha_{4}=\frac{\gamma}{(1+\gamma)^{4}},
	\]
	\[
	\beta_{1}=\gamma,\quad
	\beta_{2}=-\frac{\gamma}{(1+\gamma)^{2}},\quad
	\beta_{3}=\frac{2\gamma+1}{(1+\gamma)^{2}},\quad
	\beta_{4}=-\frac{\gamma^{2}}{(1+\gamma)^{4}} .
	\]
	Then $\sigma=\tau=2(1+\gamma)$ and $\Delta=4(1+\gamma)$, and
	\eqref{eq:abclosed} gives
	\[
	a_{\mathrm{pol}}=\gamma,
	\qquad
	b_{\mathrm{pol}}=-\frac{\gamma(\gamma-1)}{(1+\gamma)^{2}} ,
	\]
	in agreement with \eqref{eq:ab2} through
	Corollary~\ref{cor:orbital}, since $g$ takes the value
	$(1+\gamma)^{2}/\gamma$ at the equilibrium and
	$a_{\mathrm{pol}}=g^{2}a$, $b_{\mathrm{pol}}=gb$.  The degeneracy
	condition of \eqref{eq:zeros}, namely
	$x_{0}\beta_{3}\sigma+y_{0}\alpha_{3}\tau=0$, holds precisely at
	$\gamma=1$.  Finally $\sigma\neq0$ and
	$\alpha_{2}\beta_{2}-\alpha_{3}\beta_{3}\neq0$ along the whole curve,
	so no point of the family is of the non-isolated type of
	Proposition~\ref{prop:branches}, as Corollary~\ref{cor:equiv} requires.
	The degenerate point $\gamma=1$ is therefore of the kind described in
	Proposition~\ref{prop:classification}(ii): $\sigma=\sigma^{(2)}$,
	$b=0$, $a\neq0$, and the equilibrium is isolated of multiplicity two.
	
	\section{Conclusions}
	
	The Bogdanov--Takens coefficients are usually introduced through a
	coordinate-dependent normal-form reduction. Theorem~\ref{thm:A} shows that,
	for planar nilpotent singularities of rank one, they admit a substantially
	simpler intrinsic description: they are exactly the directional derivatives,
	along the kernel of the linearization, of the determinant and the trace of the
	Jacobian. Consequently, the two coefficients are determined entirely by the
	intrinsic invariants of the linear part together with the distinguished kernel
	direction, without the use of generalized eigenvectors or second-order
	multilinear forms.
	
	This intrinsic identity also provides a geometric interpretation of the
	Bogdanov--Takens nondegeneracy conditions. Along the equilibrium manifold of a
	parametrized family, the coefficients $a$ and $b$ measure the relative position
	of the kernel line field with respect to two natural hypersurfaces: the fold
	set and the neutrality set. The two possible losses of nondegeneracy therefore
	acquire an invariant geometric meaning, replacing the coordinate computations
	through which they are traditionally introduced.
	
	The second contribution is of a different nature. For Kolmogorov systems, the
	order of $f$ in the Takens normal form is identified with the multiplicity of
	the equilibrium and bounded above by the mixed volume of the associated Newton
	polytopes. Consequently, the obstruction to the vanishing of $a$ is not a
	model-dependent algebraic coincidence but a combinatorial property of the
	supports of the defining polynomials. The affine-bilinear class appears as the
	first nontrivial realization of this principle, characterized by the extremal
	value $\MV=2$. Equally important is the limitation of the result: the mixed
	volume constrains $\ordop f$, and therefore the first Bogdanov--Takens
	coefficient, but contains no information about $\ordop g$. The methods
	developed here thus reveal a genuine asymmetry between the two coefficients:
	one is controlled by combinatorial invariants of the germ, whereas the other
	remains beyond the reach of the present approach.
	
	The ecological families studied in
	Section~\ref{sec:models} serve as representative illustrations of the general
	theory rather than as its motivation. Their Bogdanov--Takens curves exhibit
	exactly the behavior predicted by the abstract results: the coefficient $a$
	never vanishes, each family possesses a unique degenerate point where $b=0$,
	and the intrinsic, geometric and combinatorial descriptions agree. At those
	points the singularity is a cusp of codimension at least three; as explained in
	Remark~\ref{rem:exactcodim}, the exact codimension depends on $\ordop g$ and
	therefore cannot be decided from the mixed-volume bound alone.
	
	Taken together, these results show that Bogdanov--Takens nondegeneracy admits
	three complementary descriptions. At the intrinsic level it is encoded by the
	directional derivatives of the Jacobian invariants; at the geometric level by
	the relative position of the kernel line field with respect to the fold and
	neutrality hypersurfaces; and, for Kolmogorov systems, at the combinatorial
	level by the Newton polytopes and their mixed volume. The convergence of these
	three viewpoints provides a unified framework in which algebraic, geometric and
	combinatorial structures describe complementary aspects of the same local
	bifurcation phenomenon while clearly identifying the boundary of the present
	theory.
	
	\section*{Acknowledgements}
	E. Chan-L\'opez acknowledges support from SECIHTI through the	``Estancias Posdoctorales por M\'exico'' program (CVU 422090). The author thanks A. Mart\'in-Ruiz (ICN-UNAM) for valuable comments
	and suggestions that contributed to the final version of the manuscript.
	
	\section*{Ethics declarations}
	\subsection*{Conflict of interest}
	The authors declare no conflicts of interest.
	
	\subsection*{Ethical Approval}
	Not applicable.
	
	\section*{Supplementary Material}
	The accompanying Wolfram Mathematica Notebook provides a self-contained
	symbolic verification of the results, recomputing all 81 assertions
	from scratch in exact arithmetic.

\end{document}